\documentclass{amsart}

\usepackage{pinlabel}
\usepackage[all]{xy}
\usepackage{amscd}

\usepackage{url}
\usepackage{longtable}
\usepackage{fancyvrb}
\usepackage{hyperref}

\title{Infinite families of standard Cappell-Shaneson spheres}
\author{Kazunori Iwaki}
\email{iwaki.k.ab@m.titech.ac.jp}

\def\SL{SL(3;\mathbb{Z})}

\newcommand{\ztheta}[1]{\mathbb{Z}[\theta_{#1}]}
\newcommand{\Mod}[1]{\ (\mathrm{mod}\ #1)}

\theoremstyle{plain}
\newtheorem{theorem}{Theorem}[section]
\newtheorem{proposition}[theorem]{Proposition}
\newtheorem{corollary}[theorem]{Corollary}
\newtheorem{lemma}[theorem]{Lemma}
\newtheorem{conjecture}[theorem]{Conjecture}

\theoremstyle{definition}
\newtheorem{definition}[theorem]{Definition}
\newtheorem{example}[theorem]{Example}

\theoremstyle{remark}
\newtheorem{remark}[theorem]{Remark}

\begin{document}

\begin{abstract}
    Cappell-Shaneson homotopy 4-spheres (CS spheres) are potential counterexamples of the smooth 4-dimensional Poincar\'{e} conjecture. Akbulut proved that infinite CS spheres are diffeomorphic to the standard 4-sphere by Kirby calculus. Kim and Yamada found another family of CS spheres which is composed of standard CS spheres. In this paper, we prove more CS spheres are standard. We give 145 new infinite families of CS spheres which are diffeomorphic to the standard 4-sphere.
\end{abstract}
\maketitle

\section{Introduction}\label{section:intro}
    The smooth 4-dimensional Poincar\'{e} conjecture is one of the most important problems in differential topology:

    \begin{conjecture}[The smooth 4-dimensional Poincar\'{e} conjecture]
    Every homotopy 4-sphere is diffeomorphic to~$S^4$.
    \end{conjecture}

	Since the topological 4-dimensional Poincar\'{e} conjecture was solved by Freedman in 1982 \cite[Theorem 1.6]{Freedman:1982-1}, homotopy $4$-sphere is homeomorphic to $S^4$. As for the smooth Poincar\'{e} conjecture in other dimension, the number of exotic $n$-spheres ($n \neq 4$) are well understood. The table~\ref{table:difst} shows the number of differential structures on $S^n$ for $1 \leq n \leq 9$. Please read Kervaire and Milnor \cite{Kervaire-Milnor:1963-1} for background reading.

    \begin{table}[h]\label{table:difst}
    
    \centering
    \caption{the number of differential structures on $S^n$}
    
    \begin{tabular}{c|r|r|r|r|r|r|r|r|r}
    
    n & 1 & 2 & 3 & 4 & 5 & 6 & 7 & 8 & 9\\ \hline
      & 1 & 1 & 1 & ? & 1 & 1 & 28 & 2 & 8\\ 
    \end{tabular}
    
    \end{table}

	The smooth 4-dimensional Poincar\'{e} conjecture is the only open problem in many versions of Poincar\'{e} conjecture. This conjecture has many potential counterexamples. Cappell-Shaneson homotopy 4-spheres (CS spheres) defined by Cappell and Shaneson \cite{Cappell-Shaneson:1976-1} are the most promising potential counterexamples of the smooth 4-dimensional Poincar\'{e} conjecture.
    
\subsection{History of Cappell-Shaneson spheres}\label{subsection:history} 
    CS spheres $\Sigma_A^\varepsilon$ are defined by following two things.
    \begin{itemize}
	\item A Cappell-Shaneson matrix (CS matrix) $A$.\\
	That is, $A \in \SL$ with $\det (A-I) = 1$.
	\item A choice of framing $\varepsilon \in \mathbb{Z}/2\mathbb{Z}$ ($\varepsilon = 0,1$).
	\end{itemize}
    Let $X_{c,d,n}$ be the following CS matrix.
    \[
    \begin{bmatrix}
		0 & a & b \\
		0 & c & d \\
		1 & 0 & n-c\\
		\end{bmatrix}
    \]
    Let $A_n = X_{1,1,n+2}$. The CS spheres corresponding to $A_n$ have been thoroughly studied because that subfamily was thought to be the simplest CS spheres. Here is a chronological list of the main results towards showing $\Sigma_{A_n}^\varepsilon$ are standard using Kirby calculus.

	\begin{itemize}
	\item In 1976, Cappell and Shaneson \cite[Section 3.]{Cappell-Shaneson:1976-1} defined CS spheres .
    \item In 1984, Aitchison and Rubinstein \cite[Theorem 4.3.]{Aitchison-Rubinstein:1984-1} proved $\Sigma_{A_n}^0$ is standard for all integer $n$ .
    \item In 1991, Gompf \cite{Gompf:1991-1} proved $\Sigma_{A_0}^1$ is standard .
    \item In 2009, Akbulut \cite[Theorem 1.]{Akbulut:2010-1} proved $\Sigma_{A_n}^1$ is standard for all integer $n$ .
	\end{itemize}
	
    It took 30 years to show the simplest subfamily of CS spheres is all standard by Kirby calculus. Thereafter, Gompf \cite{Gompf:2010-1} found that the mechanism in the above result is same in 2010. By the observation, Gompf introduced an equivalence relation (Gompf equivalence) on the set of CS spheres and proved that if CS matrices are Gompf equivalent then corresponding CS spheres are diffeomorphic. Since CS spheres corresponding to $A_0$ are standard, if a CS matrix is Gompf equivalent to $A_0$ then the corresponding CS sphere is standard. Gompf gave the following conjectures and proved Conjecture~\ref{conjecture:Gompf} is true for trace $-6 \leq n \leq 9$ or $n = 11$.
   
    \begin{conjecture}\label{conjecture:CS}Every CS sphere is diffeomorphic to $S^4$.
    \end{conjecture}

    \begin{conjecture}[{\cite[Conjecture 3.6]{Gompf:2010-1}}] \label{conjecture:Gompf}Every CS matrix is Gompf equivalent to~$A_0$.
    \end{conjecture}
    
	For brevity of our discussion, we say \textbf{Conjecture~\ref{conjecture:Gompf} is true for trace $n$} if every Cappell-Shaneson matrix $A$ with trace $n$ is Gompf equivalent to $A_0$ following Kim and Yamada \cite[Section 1.2.]{Kim-Yamada:2017-1}.
    
	Aitchison and Rubinstein \cite[Appendix]{Aitchison-Rubinstein:1984-1} introduced a number theoretic object, an ideal class monoid, to deal with CS matrices nicely. Since there is a one-to-one correspondence between similarity classes of CS matrices with trace $n$ and an ideal class monoid, we can use ideal class monoids instead of CS matrices. Kim and Yamada \cite{Kim-Yamada:2017-1} \cite{Kim-Yamada:2017-2} proved that even more CS spheres are standard using ideal class monoids. They proved that Conjecture~\ref{conjecture:Gompf} is true for trace $-64 \leq n \leq 69$ and $n=-69, -66, 71, 74$. And they found a new infinite family of standard CS spheres different from $\Sigma_{A_n}^\varepsilon$.
	
    \begin{theorem}[{\cite[Corollary D.]{Kim-Yamada:2017-1}}]\label{theorem:49k+27}
	CS spheres corresponding to $X_{2,7,49k+27}$ is standard for any $k$. $X_{2,7,49k+27}$ is not similar to $A_n$ for any $k$ and for any $n$. In other words, they gave an infinite family of CS spheres which are diffeomorphic to $S^4$ and the family is different from $\Sigma_{A_n}^\varepsilon$.
	\end{theorem}

\subsection{Main results}\label{subsection:main_result}
	In this paper, we extend the result by Kim and Yamada. Our main result gives us infinite families of CS spheres which are diffeomorphic to $S^4$.
    \begin{theorem}[Theorem~\ref{theorem:infiniteseries}]\label{theorem:introinfiniteseries}
    Let $(c,p,n_0)$ be a solution of the simultaneous congruence equations.
	\begin{enumerate}
		\item $(2c-1) n_0 \equiv 3 c^2 -1 \Mod{p}$
        \item $(c^2 - c) n_0 \equiv c^3 - c - 1 \Mod{p^2}$
    \end{enumerate}
    Then, $X_{c, p, p^2 k + n_0}$ is not similar to $A_n$ for any integer $k$, $n$. If $n_0 \equiv n^\prime \Mod{p}$ for $n^\prime$ such that Conjecture~\ref{conjecture:Gompf} is true for $n^\prime$, the corresponding CS spheres are diffeomorphic to $S^4$.
    \end{theorem}

    We can calculate concrete $(c,p,p^2 k + n_0)$ which satisfies the condition of Theorem~\ref{theorem:introinfiniteseries}. $(c,p,p^2 k + n_0)=(2,7,49 k + 27)$ corresponds to the result by Kim and Yamada \cite[Corollary D.]{Kim-Yamada:2017-1}. In addition, we found 145 new solutions where $p>7$ by SageMath \cite{Sage:2019-1}. 146 $(c,p,n_0)$ are listed in corollary~\ref{corollary:cpn0}.

    \begin{corollary}\label{corollary:cpn0_intro}
        There are 146 $(c,p,n_0)$ such that CS spheres corresponding to $X_{c,p,p^2 k + n_0}$ are diffeomorphic to $S^4$ for all $k$ and $\varepsilon$. Moreover, those $X_{c,p,p^2 k + n_0}$ and $A_n$ are not similar for all $k$ and $n$.
    \end{corollary}
        
\subsection*{Acknowledgements}\label{subsection:acknowledgements}

    This article is established thanks to much support from my advisor Hisaaki Endo and helpful discussions with Yuichiro Taguchi and Shun'ichi Yokoyama. An unpublished manuscript I received from Shohei Yamada was very helpful in writing this article. I grateful and would like to thank them.

\section{Preliminaries}\label{section:preliminaries}
\subsection{Construction}\label{subsection:construction}
    In this subsection, we construct CS spheres and introduce elementary facts about CS matrices.
    
	Let $T^3 = \mathbb{R}^3/\mathbb{Z}^3$. $A\in \SL$ induces a diffeomorphism $f_A\colon T^3\to T^3$. We can assume that $f_A$ is the identity on a neighborhood $D_y$ of a chosen point $y\in T^3$ after an isotopy. Let $W_A$ be the mapping torus of $f_A$.
	\[W_A=T^3\times \mathbb{R}/(x,t)\sim(f_A(x),t-1).\]
    Let $\Sigma_A^\varepsilon$ be obtained from $W_A$ by surgery on a circle $C=[{y}\times \mathbb{R}]$ with a framing $\varepsilon\in \mathbb{Z}/2\mathbb{Z}$. 
	By adding a condition, $\Sigma_A^\varepsilon$ becomes a homotopy 4-sphere.

    \begin{proposition}[{\cite[Section 3.]{Cappell-Shaneson:1976-1}, \cite[Proposition 3.1.]{Issa:2017-1}}]
    $\Sigma_A^\varepsilon$ is a homotopy 4-sphere if and only if $\det (A-I)= \pm 1$.
    \end{proposition} 
 
	By the following remark, we may assume $\det (A-I)= 1$ when we think diffeomorphism type of $\Sigma_A^\varepsilon$.
 
    \begin{remark}\label{remark:inverse}
    Let $A\in \SL$ such that $\det (A-I)= -1$. Then, $A^{-1} \in \SL$ and $\det (A^{-1}-I)= 1$. So, $\Sigma_{A^{-1}}^\varepsilon$ is a homotopy 4-sphere, which is homotopy equivalent to $\Sigma_A^\varepsilon$. Actually, $\Sigma_{A^{-1}}^\varepsilon$ is diffeomorphic to $\Sigma_A^\varepsilon$. Define $\phi : W_A \to W_{A^{-1}}$ to be $\phi([x,t])=[x,1-t]$, then $\phi$ is a diffeomorphism and induces a diffeomorphism from $\Sigma_A^\varepsilon$ to $\Sigma_{A^{-1}}^\varepsilon$.
    \end{remark}
    
    \begin{definition}
    We say a matrix $A\in \SL$ is a \textbf{Cappell-Shaneson matrix (CS matrix)} if $\det (A-I)=1$.
    For a CS matrix $A$, $\Sigma_A^\varepsilon$ is called a \textbf{Cappell-Shaneson homotopy 4-sphere (CS sphere)} corresponding to $A$.
    \end{definition}
    
    \begin{remark}\label{remark:similarimpliesdiffeomorphic}
    We say two matrices $A, B\in \SL$ are \textbf{similar} ($A \sim_S B$) if there is a matrix $P\in \SL$ such that $B=PAP^{-1}$. For CS matrices $A,B$, if $A,B$ are similar, $\Sigma^\varepsilon_A$ is diffeomorphic to $\Sigma^\varepsilon_B$. Define $\phi : W_A \to W_B$ to be $\phi([x,t])=[Px,t]$, then $\phi$ is a diffeomorphism and induces a diffeomorphism from $\Sigma_A^\varepsilon$ to $\Sigma_B^\varepsilon$.
    \end{remark}
    
        By this remark, we can focus on the similarity classes of CS matrices to study CS spheres up to diffeomorphism. We introduce standard CS matrices, which are representative elements of similarity classes of CS matrices.

    \begin{definition}[{\cite[Definition 2.5]{Kim-Yamada:2017-1}}]\label{definition:standard}We say a CS matrix is \textbf{standard} if it can be written as  
    \[X_{c,d,n}=
		\begin{bmatrix}
		0 & a & b \\
		0 & c & d \\
		1 & 0 & n-c\\
		\end{bmatrix}.
		\]
    \end{definition}
    \begin{remark}[{\cite[Remark 2.8.]{Kim-Yamada:2017-1}}]\label{remark:CS}$X_{c,d,n}$ is uniquely determined by $c,d$ and $n$ since $a,b$ are uniquely determined by the conditions $\det(X_{c,d,n})=1$ and $\det(X_{c,d,n}-I)=1$.
    \end{remark}

    Let $A$ be a CS matrix with trace $n$. The characteristic polynomial of $A$ is 
	\[f_n(x) =x^3 -n x^2 +(n-1)x -1.\]
	
	\begin{lemma}[{\cite[LEMMA A4.]{Aitchison-Rubinstein:1984-1}}]\label{lemma:irr}
	    $f_n(x)$ is irreducible over $\mathbb{Q}$ for all $n$.
	\end{lemma}
    \begin{lemma}\label{lemma:characteristicpolynomial}For $A$ be a $3\times 3$ integral matrix, $A$ is a CS matrix with trace $n$ if and only if $f_n(A) = O$.
    \end{lemma}

    \begin{proof}
    Let $A$ be a $3\times 3$ integral matrix and $f_n(A) = O$. The minimal polynomial of $A$ divides $f_n(x)$. This implies that the minimal polynomial of $A$ is equal to $f_n(x)$ since $f_n(x)$ is irreducible over $\mathbb{Q}$. And the characteristic polynomial of $A$ is equal to $f_n(x)$. This means $A$ is a CS matrix. For the converse, note that the characteristic polynomial of a CS matrix with trace $n$ is $f_n(x)$. $f_n(A) = O$ holds by the Cayley Hamilton theorem.\\
    \end{proof}

    \begin{proposition}[{\cite[Proposition 2.10.]{Kim-Yamada:2017-1}}]\label{proposition:csm2}
	For integers $c$ and $d\neq 0$ and $n$, the following are equivalent:
	\begin{enumerate}
		\item $f_n (c) \equiv 0 \Mod{d}$,
		
		\item There exist integers $a$ and $b$ such that
		\[X_{c,d,n}=
		\begin{bmatrix}
		0 & a & b \\
		0 & c & d \\
		1 & 0 & n-c \\
		\end{bmatrix}
		\]
		is a standard CS matrix.
	\end{enumerate}
    \end{proposition}

    \begin{remark}\label{remark:d0}
    Let $X_{c,d,n}$ be a CS matrix. If $d = 0$, $f_n(c) = -ad = 0$ holds. This contradicts to irreducibility of $f_n(x)$. Therefore, $d \neq 0$.
    \end{remark}

    \begin{remark}\label{remark:numberofCSmatrices}
    Since $f_{n+2} (1) \equiv 0 \Mod{1}$, $A_n=X_{1,1,n+2}$ is a standard CS matrix for any $n$. Hence, there exist infinitely many standard CS matrices and there exists a CS matrix for any trace.
    \end{remark}

    By the following theorem, we can focus on standard CS matrices to study CS spheres up to diffeomorphism.
    \begin{theorem}[{\cite[Theorem A3.]{Aitchison-Rubinstein:1984-1}}]\label{theorem:csm} Every CS matrix is similar to a standard CS matrix. 
    \end{theorem}

\subsection{Gompf equivalence}\label{subsection:Gompfequivalence}
    After it is proved that the simplest CS spheres $\Sigma_{A_n}^\varepsilon$ are standard, Gompf found that the key mechanism in the proof. Gompf introduced an equivalence relation (Gompf equivalence) on the set of CS matrices and proved that if CS matrices are Gompf equivalent then corresponding CS spheres are diffeomorphic. We recall Gompf's result briefly. 
		
	Let $\Delta$ be the following matrix,
		\[
		\Delta =
		\begin{bmatrix}
		1 & -1 & 0 \\
		0 & 1  & 0 \\
		0 & 1  & 1 \\
		\end{bmatrix}.
		\]
	\begin{theorem}[{\cite[page 1673]{Gompf:2010-1}}]\label{theorem:Gompf}
		Let $X_{c,d,n}$ be a standard CS matrix. Then, $X_{c,d,n} \Delta^k$ and $\Delta^k X_{c,d,n}$ are also CS matrices and corresponding CS spheres \\$\Sigma ^\varepsilon _{X_{c,d,n} \Delta^k}$ and  $\Sigma^\varepsilon_{\Delta^k X_{c,d,n}}$ are diffeomorphic to $\Sigma ^\varepsilon _{X_{c,d,n}}$ for every integer $k$ and $\varepsilon\in \mathbb{Z}/2\mathbb{Z}$.
	\end{theorem}

    \begin{remark}[{\cite[Remark 2.17.]{Kim-Yamada:2017-1}}]\label{remark:Gompf}
        $X_{c,d,n}\Delta^k$ and $\Delta^k X_{c,d,n}$ are similar to $X_{c,d,n+kd}$. $\Sigma^\varepsilon_{X_{c,d,n}}$ are diffeomorphic to $\Sigma ^\varepsilon _{X_{c,d,n+kd}}$ by Remark~\ref{remark:similarimpliesdiffeomorphic} and Theorem~\ref{theorem:Gompf}.
    \end{remark}
    
    We introduce simple notation following \cite[Section 2.4.]{Kim-Yamada:2017-1}. Consider 
	\[\mathcal{C}\mathcal{S}=\{(c,d,n)\in \mathbb{Z}^3\mid f_n(c)\equiv0\Mod{d}\textrm{ and }d\neq 0\}.\]
    By Proposition~\ref{proposition:csm2}, there is a bijection between $\mathcal{C}\mathcal{S}$ and the set of standard CS matrices such that $(c,d,n) \mapsto X_{c,d,n}$. We can define Gompf equivalence on $\mathcal{C}\mathcal{S}$.
    
    \begin{definition}[Gompf equivalence]\label{definition:Gompfequivalence}  
        Gompf equivalence $\sim$ is an equivalence relation on $\mathcal{C}\mathcal{S}$ generated by the following two relations,$\sim_S$ and $\sim_G$:
        \begin{enumerate}
            \item $(c,d,n) \sim_S (c^\prime,d^\prime,n^\prime)$ if and only if $X_{c,d,n}$ is similar to $X_{c^\prime,d^\prime,n^\prime}$ for any $(c,d,n),(c^\prime,d^\prime,n^\prime) \in \mathcal{C}\mathcal{S}$.
            \item $(c,d,n) \sim_G (c,d,n+kd)$ for any integer $k$ and any $(c,d,n) \in \mathcal{C}\mathcal{S}$.
        \end{enumerate}
    \end{definition}
	
	\begin{remark}\label{remark:Gompfequivalence}
	    Gompf equivalence on $\mathcal{C}\mathcal{S}$ induces Gompf equivalence on the set of the standard CS matrices. If two CS matrices are Gompf equivalent, corresponding CS spheres are diffeomorphic by Remark~\ref{remark:similarimpliesdiffeomorphic} and Remark~\ref{remark:Gompf}.
	\end{remark}

    We can reformulate Conjecture~\ref{conjecture:Gompf} using $\mathcal{C}\mathcal{S}$.
    
    \begin{conjecture}[{\cite[Conjecture 2.20.]{Kim-Yamada:2017-1}}]\label{conjecture:reformulation}For every $(c,d,n)\in \mathcal{C}\mathcal{S}$, $(c,d,n)\sim (1,1,2)$.
    \end{conjecture}

\subsection{Latimer-MacDuffee-Taussky correspondence}\label{subsection:Latimer-MacDuffee-Taussky}
    In this subsection, we recall a result by Latimer-MacDuffee and Taussky \cite{Latimer-MacDuffee:1933-1,Taussky:1949-1} in order to calculate similarity classes of CS matrices systematically.
        
    Let $R$ be an integral domain. We can define an equivalence relation $\approx$ on the set of non zero ideals of $R$, $\mathcal{I}(R)$: $I \approx J$ if and only if there exists non zero elements $\alpha,\beta \in R$ such that $\alpha I = \beta J$. Here, the general multiplication of ideals induces an operation to $\mathcal{I}(R) / \approx$. This set is a monoid and the identity element is the equivalence class of $R$. This monoid is said to be an ideal class monoid (ICM) or an ideal class semigroup. We write the ideal class monoid as $C(R)$. $I \in \mathcal{I}(R)$ is said to be \textbf{invertible} when $[I] \in C(R)$ is an invertible element in $C(R)$.
    
	If $R=\mathbb{Z}[\theta]$ where $\theta$ is a root of a monic polynomial $f(x)$ which is irreducible over $\mathbb{Q}$, we can apply the following correspondence to ICM.
    
    \begin{theorem}[Latimer-MacDuffee \cite{Latimer-MacDuffee:1933-1}, Taussky \cite{Taussky:1949-1}]\label{theorem:LMT}
    	Suppose $f \in \mathbb{Z}[x]$ is a monic polynomial of degree $n$ and irreducible over $\mathbb{Q}$. Let  $\theta$ be a root of $f$.
    	Then there is a bijection between $C(\ztheta{})$ and $\{A \in M(n;\mathbb{Z}) | f(A) = O\}$/$\sim_S$.
    \end{theorem}
    
	Applying Theorem~\ref{theorem:LMT} to CS matrices enables us to deal with standard CS matrices systematically.
    \begin{corollary}[{\cite[page 44]{Aitchison-Rubinstein:1984-1}}]\label{corollary:Aitchison-Rubinstein}
		Let $\theta_n$ be a root of $f_n(x) = x^3 - n x^2 + (n-1) x -1$, which is the characteristic polynomial of CS matrices with trace $n$.
        There is a one-to-one correspondence between the set of similarity classes of 
		CS matrices with trace $n$ and $C(\ztheta{n} )$ such that
		\[X_{c,d,n}=
		\begin{bmatrix}
		0 & a & b \\
		0 & c & d \\
		1 & 0 & n-c \\
		\end{bmatrix}
		\mapsto 
		[\langle \theta_n -c,d \rangle ]
		\]
		
		where $f_n(c)\equiv0\Mod{d}$, $b=(c-1)(n-c-1)$ and $ad-bc=1$.
	\end{corollary}
        
    \begin{remark}[{\cite[Remark 2.15.]{Kim-Yamada:2017-1}}]\label{remark:changenbyc}
    We easily see that $\langle\theta_n-c,d\rangle=\langle\theta_n-c-kd,d\rangle$. By by Corollary~\ref{corollary:Aitchison-Rubinstein}, $X_{c,d,n}$ and $X_{c+kd,d,n}$ are similar.
    \end{remark}

    \begin{remark}
        $[\langle \theta_{n+2} -1,1 \rangle]$ is $[\ztheta{n}]$, which is the identiy element. Therefore, a CS matrix $A_n=X_{1,1,n+2}$ corresponds to the identity element in $C(\ztheta{n} )$.
    \end{remark}

\subsection{Symmetry between Cappell-Shaneson matrices}\label{subsection:symmetry}
    In this subsection, we introduce symmetry between CS matrices. Kim and Yamada \cite{Kim-Yamada:2017-1} proved this by proving symmetry in algebra and sending it to the world of CS matrices.

\begin{theorem}[{\cite[Theorem 3.3]{Kim-Yamada:2017-1}}]\label{theorem:sym1}
		There is a bijection between the set of similarity classes of 
		CS matrices with trace $n$ and the set of similarity classes of CS matrices with trace $5-n$, 
		which is explicitly defined by
		\[A=
		\begin{bmatrix}
		0 & a & b \\
		0 & c & d \\
		1 & 0 & n-c \\
		\end{bmatrix}		\mapsto A^*=
		\begin{bmatrix}
		0 & a^*& b^* \\
		0 & c^* & d^* \\
		1 & 0 & 5-n-c^* \\
		\end{bmatrix}
		\]
		where $c^*=p_n(c)=c^2+(1-n)c+1$, $d^*=d$. In particular, $X_{c,d,n}^*=X_{p_n(c),d,5-n}^{\vphantom{*}}$.
	\end{theorem}
This map is compatible with Gompf equivalence.

\begin{lemma}[{\cite[Theorem 3.5]{Kim-Yamada:2017-1}}]\label{lemma:Gompfdual}Suppose that $A$ and $B$ are two standard CS matrices such that $A$ and $B$ are Gompf equivalent. Then $A^*$ and $B^*$ are also Gompf equivalent.
\end{lemma}

Theorem~\ref{theorem:sym1} and Lemma~\ref{lemma:Gompfdual} concludes the following theorem.

 \begin{theorem}[{\cite[Theorem A.]{Kim-Yamada:2017-1}}]\label{theorem:sym2}Conjecture~\ref{conjecture:Gompf} is true for trace $n$ if and only if Conjecture~\ref{conjecture:Gompf} is true for trace $5-n$ for any integer $n$. 
 \end{theorem}

\begin{remark}
Symmetry also holds for $\ztheta{n}$,$\mathbb{Q}(\theta_n)$,$C(\ztheta{n})$\cite[Section 3.]{Kim-Yamada:2017-1}.
\begin{itemize}
		\item $\ztheta{n}$ is isomorphic to $\ztheta{5-n}$.
		\item $\mathbb{Q}(\theta_n)$ is isomorphic to $\mathbb{Q}(\theta_{5-n})$.
        \item $C(\ztheta{n})$ is isomorphic to $C(\ztheta{5-n})$.
	\end{itemize}
\end{remark}

\section{Ideal class monoid \texorpdfstring{$C(\ztheta{n})$}{C(Z[theta])}}\label{section:idealclassmonoid}

	In this section, we characterize when ideal classes in $C(\mathbb{Z}[\theta_n])$ is not invertible. If an ideal class is not invertible, the Cappell-Shaneson matrix corresponding to it is not similar to $A_n$, which corresponds to the invertible ideal class.

	\subsection{Dedekind-Kummer theorem}\label{subsection:Dedekind-Kummer}
In this subsection, we recall Dedekind-~\\Kummer theorem following \cite{Kim-Yamada:2017-1} to characterize when an ideal $\langle \theta_n-c,d\rangle$ is invertible.

	\begin{definition}[order]A \textbf{number field} $K$ is a finite degree field extension of $\mathbb{Q}$. For a number field $K$ with degree $n$, a subring $R$ of the number field $K$ is called an \textbf{order} if $R$ is a free $\mathbb{Z}$-module of rank $n$.
	\end{definition}
	
	\begin{example} Let $\theta_n$ be a root of the monic, irreducible polynomial $f_n(x)=x^3-nx^2+(n-1)x-1$. $\mathbb{Z}[\theta_n]$ is an order in the number field $\mathbb{Q}(\theta_n)$ 
    \end{example}

    \begin{remark}If $R$ is the ring of integers of an algebraic number field, then $C(R)$ is isomorphic to the ideal class group. We can understand ICM as an generalization of the ideal class group.
	\end{remark}
    \begin{definition}[integrally closed]
    	Let $R$ be an integral domain and $K$ the fraction field of $R$. We say that $R$ is \textbf{integrally closed} if the following condition holds: if $\alpha \in K$ is a root of a monic in $R[x]$, then $\alpha \in R$.
    \end{definition}

    \begin{definition}[Dedekind domain]
    	Let $R$ be an integral domain. We say that $R$ is a \textbf{Dedekind domain} if the following conditions hold:
    \begin{enumerate}
        \item $R$ is a Noetherian ring,
        \item $R$ is not a field and every non-zero prime ideal of $R$ is a maximal ideal,
        \item $R$ is integrally closed.
    \end{enumerate}
    \end{definition}
	\begin{theorem}[{\cite[Sections 6--7]{Stevenhagen:2008-1}}]\label{theorem:Dedekind} Let $\mathcal{O}_K$ be the ring of integers of $K$. For an order $R\subset K$, the following conditions are equivalent:
	\begin{enumerate}
	\item $R$ is integrally closed,
	\item  $R$ equals to $\mathcal{O}_K$, 
	\item $R$ is a Dedekind domain,
	\item Every ideal of $R$ is invertible,
	\item $C(R)$ is a group.
	\end{enumerate}
	\end{theorem}

Dedekind-Kummer theorem gives us characterization when ideals of the form $\langle \theta_n-c,p\rangle$ are invertible when $p$ is a prime. 	
	
	\begin{proposition}[{\cite[Proposition 4.7.]{Kim-Yamada:2017-1}}]\label{proposition:invertible-prime}Suppose that integers $c$, $n$ and $p$ satisfy $f_n(c)\equiv 0 \Mod {p}$. If $p$ is prime, then $\langle \theta_n-c,p\rangle$ is a prime ideal of $\mathbb{Z}[\theta_n]$. The ideal $\langle \theta_n-c,p\rangle$ is invertible 
	 if and only if at least one of the following conditions holds.
	\begin{enumerate}
	\item $c$ is a simple root of $f_n(x)$ modulo $p$. 
	\item $p^2$ does not divide $f_n(c)$.
	\end{enumerate}
	\end{proposition}

    \begin{remark}[{\cite[Remark 4.5.]{Kim-Yamada:2017-1}}]\label{remark:decomposition}
	Suppose that  $p$ and $q$ are relatively prime integers. Then $\theta_n-c$ is a linear combination of $p(\theta_n-c)$ and $q(\theta_n-c)$. It follows that
	\[\langle \theta_n-c,p\rangle \langle \theta_n-c,q\rangle=\langle (\theta_n-c)^2, p(\theta_n-c),q(\theta_n-c),pq\rangle =\langle \theta_n-c,pq\rangle.\]
	More generally, consider the prime factorization $d=p_1^{e_1}\cdots p_m^{e_m}$. Then 
	 \[\langle \theta_n-c,d\rangle =\langle \theta_n-c,p_1^{e_1}\rangle\langle \theta_n-c,p_2^{e_2}\rangle \cdots \langle \theta_n-c,p_m^{e_m}\rangle.\]
\end{remark}	
All representatives of ideal classes in $C(\ztheta{n})$ can be written as $\langle \theta_n-c,d\rangle$ by Corollary~\ref{corollary:Aitchison-Rubinstein}. Remark~\ref{remark:decomposition} shows that we can judge $\langle \theta_n-c,d\rangle$ is invertible if we know all ideal class with $d=p^e$ is invertible or not.

	The following proposition gives us a condition when $\langle \theta_n-c,p^e\rangle$ is invertible.
	
	\begin{proposition}[{\cite[Proposition 4.8]{Kim-Yamada:2017-1}}]\label{proposition:invertible-primepower}Suppose that $p$ is a prime integer and an integer $c$ satisfies $f_n(c)\equiv 0\Mod {p^k}$ for some positive integer $k$. 
	\begin{enumerate}
	\item If $\langle \theta_n-c,p\rangle $ is invertible, then $\langle \theta_n-c,p^k\rangle$ is invertible. 
	\item If $f_n(c)\not\equiv 0\Mod{p^{k+1}}$, then $\langle \theta_n-c,p^k\rangle$ is invertible.
	\end{enumerate}
	\end{proposition}
	
	\subsection{When \texorpdfstring{$C(\ztheta{n})$}{C(Z[theta])} is not a group}\label{subsection:49k+27}
	In this subsection, we give the necessary and sufficient conditions for $C(\ztheta{n})$ not to be a group. We give non-invertible ideals explicitly. 
	 
     \begin{example}\label{example:example1}
    By Theorem~\ref{theorem:Dedekind}, we can check whether $C(\ztheta{n})$ is a group or not for given $n$. For example, SageMath has a command to judge if a given ring $R$ is integrally closed or not \cite{Sage:2019-1}. According to SageMath, $C(\ztheta{n})$ is not a group for following $n$ if $0 \leq n \leq 1000$. 
    \begin{equation*}
    \begin{split}
    n=&27, 76, 94, 125, 127, 159, 167, 174, 223, 235, 272, \\
    &284, 299, 321, 370, 416, 419, 440, 456, 468, 517, 566, \\
    &615, 623, 664, 705, 713, 745, 762, 764, 807, 811, 828, \\
    &860, 909, 958, 969, 975, 994
	\end{split}
	\end{equation*}
    
	And we give the SageMath code. 
	
\hrulefill    
\begin{Verbatim}
for k in range(1000):
    O.<theta> = EquationOrder(x^3-k*x^2+(k-1)*x-1)
    if not O.is_integrally_closed():  
        print(k)
\end{Verbatim}
\hrulefill    
    \end{example}
     
    By combining the fact in subsection~\ref{subsection:Dedekind-Kummer}, we can characterize $n$ such that $C(\ztheta{n})$ is not a group.
    
    \begin{theorem}\label{theorem:A}
    Let $\theta_n$ be a root of $f_n(x) = x^3 - n x^2 +(n-1) x - 1$.
    $C(\ztheta{n})$ is not a group if and only if there exist an integer $c$ and a prime number $p$ which satisfy the following simultaneous congruence equations.
    \begin{enumerate}
		\item $(2c-1) n \equiv 3 c^2 -1 \Mod{p}$
        \item $(c^2 - c) n \equiv c^3 - c - 1 \Mod{p^2}$
    
	\end{enumerate}
    Moreover, $\langle \theta_n-c,p \rangle$ is a non-invertible ideal in $C(\ztheta{n})$ when $c,p$ satisfies the above equations.
    \end{theorem}
    
    \begin{proof}
    We assume $C(\ztheta{n})$ is not a group. There exists an element in $C(\ztheta{n})$ which is not invertible. We can write the element as $\langle \theta_n -c, d \rangle$ by Corollary~\ref{corollary:Aitchison-Rubinstein}. $d|f_n(c)$ holds by the property of standard CS matrices. Besides, Remark~\ref{remark:decomposition} says that there exist $\langle \theta_n -c, p^k \rangle$ which is not invertible and $p^k | f_n(c)$ holds where $p$ is a prime number. By Proposition~\ref{proposition:invertible-primepower}, there exists $\langle \theta_n -c, p \rangle$ which is not invertible and $p | f_n(c)$ holds. $c$ is a multiple root of $f_n(x)$ $\Mod{p}$ and $p^2 | f_n(c)$ from Proposition~\ref{proposition:invertible-prime}. We use the fact that $c$ is a multiple root of $f_n(x)$ $\Mod{p}$ if and only if $f_n^\prime(c) \equiv 0 \Mod{p}$. Then, we get $(2c-1) n \equiv 3 c^2 -1 \Mod{p}$ and $(c^2 - c) n \equiv c^3 - c - 1 \Mod{p^2}$.
    
    Conversely, we assume that there exist an integer $c$ and a prime number $p$ which satisfy the following simultaneous congruence equations.
    \begin{enumerate}
		\item $(2c-1) n \equiv 3 c^2 -1 \Mod{p}$.
        \item $(c^2 - c) n \equiv c^3 - c - 1 \Mod{p^2}$.
    \end{enumerate}
    We use the fact that $c$ is a multiple root of $f_n(x)$ $\Mod{p}$ if and only if $f_n^\prime(c) \equiv 0 \Mod{p}$. Then, the following conditions hold.
    	\begin{enumerate}
		\item $c$ is a multiple root of $f_n(x)$ $\Mod{p}$.
        \item $p^2 | f_n(c)$.
	\end{enumerate}
	By Proposition~\ref{proposition:invertible-primepower}, $\langle \theta_n -c, p \rangle$ is not invertible and $p | f_n(c)$ hold. This means that $C(\ztheta{n})$ is not a group.
    \end{proof}
    
	In \cite{Kim-Yamada:2017-1}, Kim and Yamada proved that $C(\mathbb{Z}[\theta_{49k+27}])$ is not a group for any integer $k$. We find that this ``arithmetic sequence'' structure is not special. The following corollary holds.
    
    \begin{corollary}\label{corollary:p2}
    If $C(\ztheta{n_0})$ is not a group, then there exists a prime number $p$ such that $C(\ztheta{n_0 + p^2 k})$ is not a group for any integer $k$.
    \end{corollary}
    
    \begin{proof}
    We assume $C(\ztheta{n_0})$ is not a group. Then, there exist $c\in \mathbb{Z}$ and a prime number $p$ such that $(2c-1) n_0 \equiv 3 c^2 -1 \Mod{p}$ and $(c^2 - c) n_0 \equiv c^3 - c - 1 \Mod{p^2}$. In this situation, for any integer $k$, $(2c-1) (n_0 + p^2 k) \equiv 3 c^2 -1 \Mod{p}$ and $(c^2 - c) (n_0 + p^2 k) \equiv c^3 - c - 1 \Mod{p^2}$ holds. This means $C(\ztheta{n_0 + p^2 k})$ is not a group by the Theorem~\ref{theorem:A}.
    \end{proof}

\begin{example}\label{example:example2}
	For example, the following tuples is first 10 solutions of the equations in ascending order for $p$. By Theorem~\ref{theorem:A}, $\langle \theta_{n_0 + p^2 k} - c, p \rangle$ is a non-invertible ideal.
    \begin{equation*}
    \begin{split}
    (c,p,n_0)=&(2, 7, 27), (13, 17, 127), (11, 17, 167), (19, 23, 235), (11, 23, 440),\\
	&(10, 23, 299), (8, 23, 94), (29, 31, 159), (11, 31, 807), (22, 41, 1402) \\ 
	\end{split}
	\end{equation*}
    And the following $n_k$ is arithmetic sequences corresponding to the above tuples. For the folloiwng $n_k$ and for any integer $k$, $C(\ztheta{n_k})$ is not a group.
    \begin{equation*}
    \begin{split}
    n_k=&7^2k +27, 17^2k + 127, 17^2k + 167, 23^2k + 235, 23^2k + 440,\\
    &23^2k + 299, 23^2k + 94, 31^2k + 159, 31^2k + 807, 41^2k + 1402
	\end{split}
	\end{equation*}

The following SageMath code gives us $(c,p,n_0,i)$ in ascending order for $p$, where $\langle \theta_{n_0} - c, p  \rangle$ is a representative of non-invertible ideal classes and $i$ is an index.

\hrulefill    
\begin{Verbatim}
i=1
for p in Primes():
    R.<c>=Integers(p)[]
    f = c^4-2*c^3+c^2+2*c-1
    L = [(-1)*Integer(list(l[0])[0]) for l in list(f.factor()) \
    if l[0].degree() == 1]
    for c0 in L:
        g = (Integers(p)(2*c0-1))*c-Integers(p)(3*c0^2-1)
       
        n0 = (-1)*Integer(list(g.factor())[0][0].list()[0])
        
        N = c0^3-c0-1-n0*(c0^2-c0)
        if N % p == 0:
            h = Integers(p)(c0^2-c0)*c-Integers(p)(N/p)
            n1 = (-1)*Integer(list(h.factor())[0][0].list()[0])
            n = n1*p+n0
            if ((2*c0-1)*n-3*c0^2+1)%p == 0 \
            and (c0^3-n*c0^2+(n-1)*c0-1)%(p^2) == 0:
                print((Integers(p)(c0),p,Integers(p^2)(n),i))
                i=i+1
\end{Verbatim}
\hrulefill

\end{example}

\section{Finding representatives of ideal classes}\label{section:representatives}
	In this section, we calculate representatives of elements of $C(\mathbb{Z}[\theta_n])$ to give specific CS spheres to consider. And we judge one by one whether the CS sphere is diffeomorphic to standard 4-sphere or not.

	Ideal classes usually have many representatives. In order to calculate the representatives, it is useful to decide which representative to choose.
    
	Let $x = \langle \theta_n-c, d \rangle$ be an element of $C(\mathbb{Z}[\theta_n])$. By Remark~\ref{remark:d0}, $d \neq 0$. And by Remark~\ref{remark:changenbyc}, we can select the representative such that $1 \leq c \leq d$. We define an order in such representatives of $x$, $(c,d,n) \leq (c^\prime, d^\prime, n)$ if either $d < d^\prime$ or $d = d^\prime$ and $c \leq c^\prime$. We say $(c,d,n)$ is the \textbf{minimal representative} of $x$ if $\langle \theta_n-c, d \rangle$ is the representative of $x$ and minimal with respect to this order. We use a software MAGMA \cite{MAGMA} to give the minimal representatives for a given trace.
\subsection{Representatives of ideal classes for given trace \texorpdfstring{$n$}{n}}\label{subsection:algorithm}
    In this subsection, we give MAGMA codes which show all minimal representatives for a given trace $n$. Kim and Yamada implemented the algorithm when $C(\mathbb{Z}[\theta_n])$ is a group \cite[Section 5.2.]{Kim-Yamada:2017-1}. We improved that program by using Marseglia's result \cite{Marseglia:2019-1} so that the algorithm works even though $C(\mathbb{Z}[\theta_n])$ is not a group. 
    In order to use the following program, one must download the MAGMA system to one's PC and attach the Marseglia's program to the MAGMA system. Marseglia's algorithm computing ICM is distributed in github 
    \url{https://github.com/stmar89/AbVarFq}. We refer the reader to 
    \url{https://magma.maths.usyd.edu.au/magma/handbook/text/24#185} on how to attach the program. If $C(\mathbb{Z}[\theta_n])$ is a group, you can compute representatives in SageMath.

\hrulefill
\begin{Verbatim}
n := 76;
R<x>:=PolynomialRing(Integers());
f:=x^3-n*x^2+(n-1)*x-1;
A:=AssociativeAlgebra(f);
E:=EquationOrder(A);
E_icm:=ICM(E);
C:=# E_icm;
X:=ZeroMatrix(IntegerRing(),2,C);

"There are", C, "similarity classes of trace", n,
"Cappell-Shaneson matrices.";

function LMT_ICM(c,d,n)
  k:=Evaluate(f,c);
  if IsDivisibleBy(k,d) eq false then 
    return "there is not a standard CS matrix corresponding
    to c,d,n";
  else
    Y:= Matrix(IntegerRing(), 3, 3,
    [0,(k/(-d)),((c-1)*(n-c-1)),0,c,d,1,0,n-c]);
    bas1:=MatrixToIdeal(A,Y);
    return ideal<E|bas1>;
  end if;
end function;


print [1,1,n];

i:=1;
d:=1;
while i lt C do
  for c in [1 .. d] do
    if IsDivisibleBy(Evaluate(f,c),d) eq true then
      I:=LMT_ICM(c,d,n);
      if IsPrincipal(I) ne true then
			
        if i ne 1 then
          IsSame := false;
          for j in [1 .. (i-1)] do
						
            if IsIsomorphic2(I,LMT_ICM(X[1][j],X[2][j],n)) then
              IsSame:=true; break;
            end if;
          end for;
          if IsSame eq false then
            X[1][i] := c; X[2][i] := d; i+:=1; print [c,d,n];
          end if;
        end if;
        if i eq 1 then
          X[1][i] := c; X[2][i] := d; i+:=1; print [c,d,n];
        end if;
      end if;
    end if;
  end for;
  d+:=1;
end while;
\end{Verbatim}
\hrulefill

		\subsection{Other representatives for given ideal classes}\label{subsection:algorithm2}
In this subsection, we give a program to find (not only minimal) representatives for given element $[\langle \theta_n - c, d \rangle]$. This program is almost the same as \cite[Section 5.2.]{Kim-Yamada:2017-1} except for using Marseglia's program.

\clearpage
\hrulefill
\begin{Verbatim}
n := 27;
c0 := 4;
d0 := 5;
N:= 1000;
R<x>:=PolynomialRing(Integers());
f:=x^3-n*x^2+(n-1)*x-1;
A:=AssociativeAlgebra(f);
E:=EquationOrder(A);

function LMT_ICM(c,d,n)
  k:=Evaluate(f,c);
  if IsDivisibleBy(k,d) eq false then 
    return "there is not standard CS matrix
    corresponding to c,d,n";
  else
    X:= Matrix(IntegerRing(), 3, 3,
    [0,(k/(-d)),((c-1)*(n-c-1)),0,c,d,1,0,n-c]);
    bas1:=MatrixToIdeal(A,X);
    return ideal<E|bas1>;
  end if;
end function;

for d in [1 .. N] do
  for c in [1 .. d] do
    if IsDivisibleBy(Evaluate(f,c),d) eq true then
      if IsIsomorphic2(LMT_ICM(c,d,n),LMT_ICM(c0,d0,n)) eq
      true then
      [c,d,n];
      end if;
    end if;
  end for;
end for;
\end{Verbatim}
\hrulefill

\subsection{Table of representatives of \texorpdfstring{$C(\ztheta{n})$}{C(Z[theta])}}\label{subsection:table}

The representatives are contained \cite[Section 5.4.]{Kim-Yamada:2017-1} for $3 \le n \le 69$. We calculated representatives for $70 \le n \le 78$.
    
\begin{center}
\begin{longtable}{c||c|l}

\caption{Representatives of elements of $C(\mathbb{Z}[\theta_n])$ for $70\leq n\leq 78$}
\label{table:representatives}\\
\hline
$n$ & $\# C(\mathbb{Z}[\theta_n])$ & Representatives of elements of $C(\mathbb{Z}[\theta_n])$ \\ 
\hline
70&44&(1,1,70), (2,3,70), (2,5,70), (5,7,70), (2,9,70),\\
&& (4,11,70), (2,15,70), (8,17,70), (5,21,70), (17,23,70),\\
&& (12,25,70), (2,27,70), (9,29,70), (6,31,70), (26,33,70),\\
&& (12,35,70), (17,41,70), (12,43,70), (22,43,70), (36,43,70),\\
&& (10,47,70), (14,47,70), (19,49,70), (8,51,70), (37,55,70),\\
&& (48,59,70), (50,59,70), (47,63,70), (8,67,70), (52,71,70),\\
&& (48,73,70), (62,75,70), (29,81,70), (68,93,70), (47,105,70),\\
&& (17,115,70), (26,121,70), (37,125,70), \underline{\underline{(104,141,70)}}, \\
&& \underline{(47,151,70)},\underline{\underline{(37,155,70)}}, \underline{(59,187,70)}, \\
&& \underline{(186,199,70)}, \underline{(96,203,70)}\\
71&21&(1,1,71), (10,11,71), (3,13,71), (4,13,71), (12,13,71),\\
&& (6,17,71), (12,19,71), (8,23,71), (9,23,71), (3,31,71),\\
&& (12,31,71), (25,31,71), (26,37,71), (33,41,71), (45,59,71),\\
&& (24,67,71), (25,67,71), (62,79,71), (31,83,71), (83,103,71),\\
&& \underline{(133,149,71)}\\
72&23&(1,1,72), (4,5,72), (4,7,72), (10,17,72), (9,19,72),\\
&& (10,19,72), (15,19,72), (4,23,72), (11,23,72), (14,25,72),\\
&& (28,29,72), (4,35,72), (17,37,72), (11,41,72), (16,43,72),\\
&& (9,47,72), (18,49,72), (28,59,72), (43,61,72), (57,67,72),\\
&& (49,79,72), (29,95,72), \underline{(144,203,72)}\\
73&38&(1,1,73), (2,3,73), (3,5,73), (6,7,73), (5,9,73),\\
&& (9,13,73), (8,15,73), (14,17,73), (4,19,73), (20,21,73),\\
&& (14,23,73), (18,25,73), (14,27,73), (23,29,73), (21,31,73),\\
&& (13,35,73), (35,39,73), (4,43,73), (23,45,73), (14,51,73),\\
&& (50,53,73), (23,57,73), (52,59,73), (41,63,73), (48,65,73),\\
&& (68,75,73), (68,81,73), (48,85,73), (63,89,73), (23,95,73),\\
&& (87,97,73), (29,103,73), (83,105,73), (50,113,73), \underline{\underline{(23,145,73)}},\\
&& \underline{\underline{(23,171,73)}}, \underline{\underline{(41,189,73)}}, \underline{\underline{(178,191,73)}}\\
74&24&(1,1,74), (2,11,74), (8,11,74), (9,11,74), (2,13,74),\\
&& (11,19,74), (13,23,74), (19,23,74), (13,37,74), (39,41,74),\\
&& (27,43,74), (7,47,74), (10,53,74), (31,53,74), (33,53,74),\\
&& (7,59,74), (52,61,74), (31,67,74), (47,67,74), (63,67,74),\\
&& (37,79,74), (25,83,74), (77,89,74), \underline{(121,191,74)}\\
75&24&(1,1,75), (2,5,75), (3,7,75), (6,13,75), (22,25,75),\\
&& (2,29,75), (4,29,75), (11,29,75), (9,31,75), (13,31,75),\\
&& (22,31,75), (17,35,75), (34,37,75), (26,41,75), (38,49,75),\\
&& (40,59,75), (14,67,75), (42,107,75), (21,113,75), (49,127,75),\\
&& (50,127,75), (103,127,75), \underline{\underline{(101,163,75)}}, \underline{\underline{(31,197,75)}}\\

76&35&(1,1,76), (2,3,76), (2,7,76), (8,9,76), (13,17,76),\\
&& (16,17,76), (6,19,76), (2,21,76), (8,27,76), (27,31,76),\\
&& (21,37,76), (35,41,76), (35,43,76), (22,47,76), (24,47,76),\\
&& (30,47,76), (9,49,76), (16,49,76), (23,49,76), (30,49,76),\\
&& (37,49,76), (44,49,76), (47,51,76), (44,57,76), (32,59,76),\\
&& (44,63,76), (35,129,76), (90,137,76), \underline{\underline{(116,141,76)}}, \\
&& \underline{\underline{(44,147,76)}},\underline{\underline{(65,147,76)}}, \underline{\underline{(86,147,76)}}, \underline{\underline{(128,147,76)}}, \\
&& \underline{\underline{(98,153,76)}}, \underline{\underline{(46,173,76)}}\\
77&35&(1,1,77), (4,5,77), (5,7,77), (6,11,77), (7,13,77),\\
&& (8,13,77), (10,13,77), (12,17,77), (15,23,77), (9,25,77),\\
&& (5,29,77), (30,31,77), (19,35,77), (14,37,77), (27,41,77),\\
&& (5,49,77), (47,53,77), (39,55,77), (14,61,77), (22,61,77),\\
&& (41,61,77), (49,65,77), (59,65,77), (40,71,77), (61,77,77),\\
&& (33,91,77), (75,91,77), (35,97,77), (55,97,77), (40,113,77),\\
&& (84,115,77), (68,127,77), \underline{\underline{(34,145,77)}}, \underline{\underline{(92,149,77)}}, \underline{\underline{(61,253,77)}}\\
78&24&(1,1,78), (3,5,78), (5,11,78), (7,17,78), (17,19,78),\\
&& (18,23,78), (23,25,78), (15,29,78), (28,31,78), (10,37,78),\\
&& (25,41,78), (21,43,78), (11,53,78), (38,55,78), (9,59,78),\\
&& (56,71,78), (9,83,78), (19,83,78), (98,125,78), (93,131,78),\\
&& (111,131,78), \underline{(35,151,78)}, \underline{(18,157,78)}, \underline{(158,185,78)}\\

\end{longtable}
\end{center}
	\section{More Cappell-Shaneson homotopy 4-spheres are standard}\label{section:maintheorem}
	
        \subsection{Argument about CS spheres with small traces}\label{subsection:theoremB}
\begin{theorem}\label{theorem:B}Conjecture~\ref{conjecture:Gompf} is true for trace $n$ if $-64\leq n\leq 69$ or $n = -73, -69, -67, -66, 71, 72, 74, 78$. \end{theorem}

In order to prove Theorem~\ref{theorem:B}, we use the following lemma.

\begin{lemma}[{\cite[Lemma 6.1.]{Kim-Yamada:2017-1}}]\label{lemma:descend}
Suppose Conjecture~\ref{conjecture:Gompf} is true for all $m$ in $3 \leq m \leq n-1$.
If $(c,d,n)$ satisfies $n \equiv n_0 \Mod{d}$ for some $n_0$ in $6-n \leq n_0 \leq n-1$, then $(c,d,n) \sim (1,1,2)$.
\end{lemma}

	This means that $(c,d,n) \sim (1,1,2)$ if $d$ is sufficiently small. If $(c,d,n)$ does not satisfy $n \equiv n_0 \Mod{d}$ for any $n_0$ such that Conjecture~\ref{conjecture:Gompf} is true for $n_0$, we say $(c,d,n)$ is \textbf{special}. To prove Theorem~\ref{theorem:B}, we judge one by one whether $(c,d,n) \sim (1,1,2)$ or not for some special $(c,d,n)$. In the table~\ref{table:representatives} in subsection~\ref{subsection:table}, special $(c,d,n)$ are underlined. Moreover, if it is uncertain whether or not a special $(c,d,n)$ is equivalent to $(1,1,2)$, such $(c,d,n)$ are double underlined.\\
	For special cases, there are the following equivalences. Some relations are found by Kim and Yamada \cite{Kim-Yamada:2017-2}. 
    
    \begin{itemize}
    \item $(47,151,70)\sim_S(149,177,70)\sim_G(149,177,-107)\sim_S(84,121,-107)\sim_G\\
    (84,121,14)$.
    \item $(59,187,70)\sim_G(59,187,-117)\sim_S(203,239,-117)\sim_G(203,239,122)\sim_S(132,203,122)\sim_G(132,203,-81)\sim_S(65,89,-81)\sim_G(65,89,8)$.
    \item $(186,199,70)\sim_G(186,199,-129)\sim_S(57,151,-129)\sim_G(57,151,22)$.
    \item $(96,203,70)\sim_S(122, 215, 70)\sim_G(122,215,-145)\sim_S(103,181,-145)\sim_G
    (103,181,36)$.
    \item $(133,149,71)\sim_G(133,149,-78)\sim_S(44,61,-78)\sim_G(44,61,-17)$.
    \item $(144,203,72)\sim_S(134,205,72)\sim_G(134,205,-133)\sim_S(52,227,-133)\sim_G
    (52,227,94)\sim_S(46,97,94)\sim_G(46,97,3)$.
    \item $(121,191,74)\sim_S(174,197,74)\sim_G(174,197,-123)\sim_S\\
    (112,191,-123)\sim_G(112,191,68)$.
    \item $(35,151,78)\sim_G(35,151,-73)\sim_S(107,131,-73)\sim_G(107,131,58)$.
    \item $(18,157,78)\sim_G(18,157,-79)\sim_S(38,137,-79)\sim_G(38,137,58)$.
    \item $(158,185,78)\sim_G(158,185,-107)\sim_S(84,143,-107)\sim_G(84,143,36)$.
	\end{itemize}
\begin{proof}[Proof of Theorem~\ref{theorem:B}]
Note that we only have to prove Theorem~\ref{theorem:B} for $n = 72, 78$ since there is symmetry \cite[Theorem A.]{Kim-Yamada:2017-1} and this theorem is proved for $-64\leq n\leq 69$ and $n = -69,-66,71,74$ \cite[Theorem 3.2.]{Gompf:2010-1} \cite[Theorem B.]{Kim-Yamada:2017-1} \cite{Kim-Yamada:2017-2}.
We can check almost all representatives are equivalent to $(1,1,2)$ by Lemma~\ref{lemma:descend}. 
The above observation completes the proof.
\end{proof}
		\subsection{Infinite series of standard CS spheres}\label{subsection:infiniteseries}
        
\begin{theorem}\label{theorem:infiniteseries}
Let $(c,p,n_0)$ be a solution of the following equations. If $n_0 \equiv n^\prime \Mod{p}$ for $n^\prime$ such that Conjecture~\ref{conjecture:Gompf} is true for $n^\prime$, $X_{c,p,p^2 k + n_0}$ is not similar to $A_n$ for any integer $k,n$ and the corresponding CS spheres are diffeomorphic to the standard 4-sphere.
    \begin{enumerate}
		\item $(2c-1) n \equiv 3 c^2 -1 \Mod{p}$
        \item $(c^2 - c) n \equiv c^3 - c - 1 \Mod{p^2}$
    
	\end{enumerate}
\end{theorem}
\begin{proof}
	By the hypothesis and Theorem~\ref{theorem:A}, $\langle \theta_{n_0+ p^2 k}-c,p \rangle$ is a non-invertible ideal. On the other hand, the ideal corresponding to $A_n$ is a principal ideal and obviously invertible. By Theorem~\ref{theorem:LMT}, we conclude that $X_{c,p,p^2 k + n_0}$ is not similar to $A_n$ for any integer $k,n$.\\
    If $n_0 \equiv n^\prime \Mod{p}$ for $n^\prime$ such that Conjecture~\ref{conjecture:Gompf} is true for $n^\prime$, $(c,p,p^2 k + n_0) \sim_G (c,p,n^\prime) \sim (1,1,2)$ holds. This completes the proof.\\
\end{proof}

	The following program gives us examples of Theorem~\ref{theorem:infiniteseries}. 
    
\hrulefill    
\begin{Verbatim}
i=1
for p in Primes():
    R.<c>=Integers(p)[]
    f = c^4-2*c^3+c^2+2*c-1
    L = [(-1)*Integer(list(l[0])[0]) for l in list(f.factor())\
    if l[0].degree() == 1]
    for c0 in L:
        g = (Integers(p)(2*c0-1))*c-Integers(p)(3*c0^2-1)
       
        n0 = (-1)*Integer(list(g.factor())[0][0].list()[0])
        
        #### checking whether (c,p,n_0) is equivalent to (1,1,2)
        if ((-64) > n0) and ((n0 + p) > 69) and\
        ((n0 + p) != 71) and ((n0 + p) != 74) and\
        (n0 != -66) and (n0 != -69) and\
        (n0 != -73) and (n0 != -67) and\
        ((n0 + p) != 72) and ((n0 + p) != 78):
            break
            
            
        N = c0^3-c0-1-n0*(c0^2-c0)
        if N % p == 0:
            h = Integers(p)(c0^2-c0)*c-Integers(p)(N/p)
            n1 = (-1)*Integer(list(h.factor())[0][0].list()[0])
            n = n1*p+n0
            if ((2*c0-1)*n-3*c0^2+1)%p == 0 and\
            (c0^3-n*c0^2+(n-1)*c0-1)%(p^2) == 0:
                print((Integers(p)(c0),p,Integers(p^2)(n),i))
                i=i+1
\end{Verbatim}
\hrulefill

    \begin{corollary}\label{corollary:cpn0}
        For the following $(c,p,n_0)$, CS spheres corresponding to $X_{c,p,p^2 k + n_0}$ are diffeomorphic to $S^4$ for all $k$ and $\varepsilon$. Moreover, $X_{c,p,p^2 k + n_0}$ and $A_n$ are not similar for all $k$ and $n$.
        
        \begin{equation*}
        \begin{split}
        (c,p,n_0)=&(2,7,27),(13, 17, 127), (11, 17, 167), (19, 23, 235), (11, 23, 440),  \\
        &(10, 23, 299),(8, 23, 94), (29, 31, 159), (11, 31, 807), (22, 41, 1402),  \\
        &(3, 41, 284),(37, 47, 975), (18, 47, 1239), (21, 73, 1405), (12, 73, 3929),  \\
        &(36, 89, 3438),(29, 89, 4488), (57, 97, 2537), (27, 97, 6877),  \\
        &(97, 103, 9341),(45, 103, 1273), (122, 127, 10211), (78, 127, 7977),  \\
        &(34, 127, 8157),(22, 127, 5923), (134, 137, 1562), (99, 137, 10386),  \\
        &(35, 137, 17212),(8, 137, 8388), (139, 151, 3643), (102, 151, 20901),  \\
        &(59, 151, 19163),(4, 151, 1905), (80, 167, 25009), (75, 167, 2885),  \\
        &(171, 199, 8020),(49, 199, 31586), (138, 223, 13651), (71, 223, 36083),  \\
        &(112, 239, 29450),(88, 239, 35909), (29, 239, 27676), (12, 239, 21217),  \\
        &(139, 241, 53010),(125, 241, 5076), (247, 257, 5395), (71, 257, 60659),  \\
        &(259, 281, 12619),(227, 311, 50323), (151, 311, 46403), (264, 313, 3111),  \\
        &(170, 313, 94863),(158, 353, 15224), (10, 353, 109390), (270, 367, 77025),  \\
        &(177, 367, 57669),(167, 433, 154186), (61, 433, 33308), (226, 479, 93903),  \\
        &(67, 479, 135543),(429, 577, 114186), (183, 577, 218748),  \\
        &(606, 647, 159187),(78, 647, 259427), (688, 751, 228364),  \\
        &(400, 761, 546433),(289, 761, 32693), (478, 769, 533704),  \\
        &(425, 769, 57662),(889, 911, 40983), (557, 929, 406861),  \\
        &(133, 929, 456185),(718, 967, 346133), (431, 977, 698619),  \\
        &(335, 977, 255915),(172, 977, 623317), (41, 977, 331217),  \\
        &(453, 1039, 961102),(178, 1039, 118424), (847, 1063, 826023),  \\
        &(339, 1063, 303951),(562, 1129, 418809), (794, 1321, 1171671),  \\
        &(664, 1321, 573375),(1015, 1361, 568965), (1009, 1361, 1283361),  \\
        &(1372, 1489, 658147),(1225, 1489, 1558979), (1043, 1553, 389743),\\
        &(700, 1553, 2022071),(1056, 1889, 3111234), (719, 1889, 457092),  \\
        &(2488, 2503, 1279092),(2264, 2777, 6659256), (1195, 2777, 1052478),  \\
        &(2899, 3847, 7836315),(735, 4567, 14107503), (57, 4567, 6749991),  \\
        &(4885, 4943, 578287),(5490, 6793, 33557371), (2340, 6793, 12587483),  \\
        &(7232, 7537, 3316292),(5937, 7537, 53490082), (6994, 7559, 27529860),  \\
        &(6131, 7559, 29608626),(7607, 8231, 19844902), (2253, 8231, 47904464),  \\
        &(3847, 8849, 33750030),(1936, 8849, 44554776), (8239, 9209, 30896240),  \\
        &(7427, 9209, 53909446),(9199, 9281, 1531299), (1231, 9281, 84605667),  \\
        &(5080, 9791, 29108701),(1495, 9791, 66754985), (9885, 12697, 457159),
        \end{split}
        \end{equation*}
        \begin{equation*}
        \begin{split}
        &(4948, 12697, 160756655),(7870, 13103, 133558858), (1821, 13103, 38129756), \\
        &(12218, 13831, 147286369),(11424, 13831, 44010197), (14972, 17327, 40475941),  \\
        &(17397, 18199, 311093683),(5191, 18199, 20109923), (23308, 24247, 581394539),  \\
        &(6785, 24247, 6522475),(21618, 32009, 964687258), (2345, 32009, 59888828),  \\
        &(56524, 91841, 2030145361),(110960, 114889, 7519485034),  \\
        &(90012, 114889, 5679997292),(122005, 142097, 18824442155),  \\
        &(383913, 566977, 296713805532),(324621, 566977, 24749113002),  \\
        &(461219, 859297, 690161571462),(183023, 859297, 48229762752),  \\
        &(490444, 1252129, 725046549615),(448374, 1252129, 842780483031),  \\
        &(1512011, 1766209, 2041158287487),(2788414, 3908497, 12639997219610),  \\
        &(2096714, 3908497, 2636351579404),(5564078, 5987777, 15849998997796),  \\
        &(4913921, 5987777, 20003474403938),\\
        &(17711432, 18378337, 204938177505669), \\
        &(14449587, 18378337, 132825093379905), \\
        &(27833855, 32455777, 673075952458623)
        \end{split}
        \end{equation*}
        
    \end{corollary}

\clearpage
\bibliographystyle{plain}
\bibliography{main}
\end{document}